\documentclass[11pt]{amsart}   
\usepackage{amssymb}
\usepackage{amsthm, amsmath}
\usepackage{epsfig}
\usepackage{ae,aecompl}
\usepackage{multicol}
\usepackage{pstricks,pst-grad,pst-node}
\usepackage[latin1]{inputenc}
\usepackage[T1]{fontenc}
\usepackage{colortbl}
\usepackage[colorinlistoftodos, textwidth=4cm, shadow]{todonotes}
\usepackage[colorlinks]{hyperref}  

\theoremstyle{plain}
\newtheorem{theorem}{Theorem}[section]
\newtheorem{lemma}[theorem]{Lemma}
\newtheorem{proposition}[theorem]{Proposition}
\newtheorem{corollary}[theorem]{Corollary}

\theoremstyle{definition}
\newtheorem{definition}[theorem]{Definition}
\newtheorem{example}[theorem]{Example}

\theoremstyle{remark}
\newtheorem{remark}[theorem]{Remark}

\def\delta{$\Delta$}
\def\M01{\mathcal{M}_{n\times m}(\{0,1\})}

\def\K{\mathbb K}

\def\lrarrow{\longrightarrow}

\def\bc#1{#1^{bc}}

\def\lar{\longrightarrow}
\def\morfismo#1{\stackrel{#1}{\lar}}

\def\reg{{\rm reg}}
\def \link{{\rm link}}
\def \Star{{\rm star}}
\def \del{{\rm del}}
\def\dd{\Delta}
\def\xx#1{\mathbf x^{#1}}
\def\comp{\hbox{\rm comp}}

\title{Regularity 3 in edge ideals associated to bipartite graphs}
\author{
Oscar Fern\'andez-Ramos
\and
Philippe Gimenez
}

\date{}

\address{\small \rm  University of Valladolid}
\email{oscarf@agt.uva.es}

\address{\small \rm  University of Valladolid}
\email{pgimenez@agt.uva.es}

\begin{document}

\footnotetext{Both authors are partially  supported by  MTM2010-20279-C02-01 (Ministerio de Ciencia e Innovaci\'on, Spain).}

\begin{abstract}
We focus in this paper on edge ideals associated to bipartite graphs and give a combinatorial characterization of those having regularity 3. When the regularity is strictly bigger than 3, we determine the first step $i$ in the minimal graded free resolution where there exists a minimal generator of degree $>i+3$, show that at this step the highest degree of a minimal generator is $i+4$, and determine the value of the corresponding graded Betti number $\beta_{i,i+4}$ in terms of the combinatorics of the associated bipartite graph. The results can then be easily extended to the non-squarefree case through polarization. We also study a family of ideals of regularity 4 that play an important role in our main result and whose graded Betti numbers can be completely described through closed combinatorial formulas.
\end{abstract}

\maketitle

%

\section*{Introduction}

Studying homological invariants of monomial ideals in a polynomial ring $R=\K [x_{1},\ldots,x_{n}]$
by looking for combinatorial properties in discrete objects (graphs, hypergraphs, simplicial complexes, \ldots) associated to them is a well known technique that has been fruitfully exploited in the last decades (see for example the surveys \cite{ha_resolutions_2006} and \cite{morey_edge_2010} and their references). In fact, it provides a quite complete dictionary between these two algebraic and combinatorial classes.

\smallskip
Classical objects used to relate combinatorics with monomial ideals are Stanley-Reisner ideals, simplicial or cellular resolutions and facet ideals. A monomial ideal generated by quadrics can be viewed as the facet ideal of a graph.
When the graph is simple, i.e., has no loops, these ideals are called {\it edge ideals} and were first introduced in \cite{villarreal_cohen-macaulay_1990}.

\smallskip
The homological invariants of a monomial ideal $I$ that we are interested in are those encoded in the minimal graded free resolution of the ideal, namely, the graded Betti numbers and the Castelnuovo-Mumford regularity. Considering the standard $\mathbb N$-grading on the polynomial ring $R$, the graded Betti number $\beta_{i,j}$ is the number of minimal generators of degree $j$ in the $i$-th syzygy module of the ideal. If we denote by $u_i$ (resp. $l_i$) the maximal (resp. minimal) degree of a minimal generator in $i$-th syzygy module then the fact that the resolution is minimal implies that $l_i\geq l_0+i$, where $l_0$ is the minimal degree of a generator of $I$ and the regularity of the ideal is defined as $\reg(I):=\max\{u_i-i\}$. An interesting situation is when all minimal generators of $I$ has same degree $l_0$ and $u_i=l_i=l_0+i$ for all $i$. In this case, $\reg(I)=l_0$ and we say that the ideal has an $l_0$-linear resolution.

\smallskip
There is a nice combinatorial characterization of edge ideals having a 2-linear resolution,
i.e., having regularity 2, in terms of the complement of the associated graph due to Fr{\"o}berg (\cite{froberg_stanley-reisner_1990}).
A new proof of this result has recently been obtained in \cite{nevo_2011} where the topology of the lcm-lattice
of edge ideals is studied.
Another proof of Fr{\"o}berg's characterization can be found in \cite{eisenbud_restricting_2005} where, moreover,
the least $i$ such that $u_i>i+2$ is characterized in a combinatorial way
when the edge ideal does not have a linear resolution.
This result was obtained independently in \cite{fernandez-ramos_first_2009} where it is also shown that $u_i=i+3$ for this value of $i$ where non-linear sygygies first appear. Moreover, $\beta_{i,i+3}$ is determined in terms of the complement of the graph. These results are recalled in the section \ref{pre} (Theorem~\ref{nuestro}) together with all the required definitions and notations. We will also show in the section that the graded Betti numbers of an arbitrary edge ideal $I$ satisfy the following property (Theorem~\ref{BettiDiagEdgeIdeal}): for every $i\geq 0$ and $j\geq i+2$,
$$
\beta_{i,j}(I)=\beta_{i,j+1}(I)=0\quad\Rightarrow\quad \beta_{i+1,j+2}(I)=0\,.
$$
It implies in particular that $u_{i+1}\leq u_i +2$ for all $i\geq 0$ (Corollary~\ref{conca}), a refinement of \cite[Theorem~5.2]{fernandez-ramos_first_2009}.

\smallskip
The aim of this paper is to characterize edge ideals associated to bipartite graphs $G$ having regularity 3 and determine, for those of regularity $>4$, the first step $i$ in the minimal resolution such that $u_i>i+3$. This will be done in section~\ref{reg3} where we will also prove that, for this value of $i$, one has that $u_i=i+4$ and show that $\beta_{i,i+4}$ is the number of induced subgraphs of the bipartite complement of $G$ that are isomorphic to cycles of minimal length. The fundamental role played by these subgraphs is the reason why we previously devote section \ref{cycle} to study the graded Betti numbers of the edge ideal associated to the bipartite complement of an even cycle. We show that such an edge ideal has regularity 4 and give closed combinatorial formulas for all its graded Betti numbers. We finally prove in the last section similar results for monomial ideals generated by quadrics that are not squarefree.

\smallskip
The dependence of the Betti numbers of edge ideals on the characteristic of the field $\K$, even in the case of edge ideals associated to bipartite graphs, prevents the possibility of obtaining similar results for higher regularity.

\section{Preliminaries}\label{pre}

\subsection{Graphs and simplicial complexes}

We start recalling elementary concepts on graphs that we will handle along this paper (see \cite{diestel_graph_2010} and \cite{munkres_elements_1984} for the terminology not included here regarding graphs and simplicial complexes, respectively).

\smallskip
Consider a finite simple graph $G$ and denote by $E(G)$ and $V(G)$ its edge and vertex sets respectively.
We say that a subgraph $H$ of $G$ is {\it induced} on a subset $V'$ of $V(G)$ if $V(H)=V'$ and $E(H)=\{\{u,v\}\in E(G) : u,v\in V'\}$. We write $H=G[V']$ or $H<G$ when
$H$ is an induced subgraph on an unspecified subset of vertices of $G$. The {\it complement} of $G$ is the graph $G^{c}$ on the same vertex set as $G$ with $E(G^{c})=\{\{u,v\}: u,v\in V(G), \{u,v\}\not\in E(G)\}$.
Given a vertex $u\in V(G)$, we denote by $N_G(u)$ the set of vertices of $G$ adjacent to $u$ and, for a subset $W\subset V(G)$, $N_{G}(W):=\bigcup_{u\in W}N_{G}(u)$. The {\it degree} of $u$, denoted
by $\deg(u)$, is the number of elements in $N_G(u)$.

\smallskip
A connected graph $G$ with $t:=|V(G)|\geq 4$ whose vertices are all of degree two is called a $t${\it -cycle} and denoted by $C_{t}$. We will say that $t$ is the {\it length} of the cycle. An induced subgraph which is also a cycle is called an {\it induced cycle}, and a graph $G$ is said to be {\it chordal} if it has no induced cycle. A graph whose vertices have all degree one has necessarily $2s$ vertices for some $s\geq 1$ and consists of $s$ disconnected edges. We denote it by $sK_2$.

\begin{definition}\label{defMatchingNumber}
The {\it induced matching number} of a graph $G$ is the maximal $s$ such that $sK_2<G$. We denote this number by $\mu(G)$.
\end{definition}

Consider now a simplicial complex $\Delta$. Given
a subset $W$ of its vertex set $V(\Delta)$, the {\it induced subcomplex of $\Delta$ on $W$} is $\Delta[W]:=\{\sigma\in \Delta : \sigma\subset W\}$. Recall that if one has two subcomplexes $\Delta_1$ and $\Delta_2$ of $\Delta$ such that $\dd=\dd_1\cup\dd_2$, there is a long exact sequence of reduced homologies,
called the {\it Mayer-Vietoris sequence},
\begin{equation}\label{MVgeneral}
 { \cdots\lrarrow \tilde H_{i}(\dd_0)\lrarrow {\begin{array}{c} \tilde H_{i}(\dd_1)\\ \oplus \\ \tilde H_{i}(\dd_2)\end{array}}\lrarrow \tilde H_{i}(\Delta)\lrarrow \tilde H_{i-1}(\dd_0)\lrarrow \cdots}
\end{equation}
whenever $\dd_0:=\dd_1\cap \dd_2\neq\emptyset$.

\begin{definition}\label{defConeSusp}
Given a simplicial complex $\Delta$
and $u,v\not\in V(\Delta)$, consider the following two simplicial complexes:
\begin{itemize}
 \item $v*\Delta:=\Delta \cup \{\{v\}\cup\sigma : \sigma\in \Delta\}$, the {\em cone} on the base $\Delta$ with appex $v$;
 \item $\Sigma_u^v\Delta:=\Delta \cup \{\{u\}\cup\sigma : \sigma\in \Delta\} \cup \{\{v\}\cup\sigma : \sigma\in \Delta\}$, the {\em suspension} of $\Delta$ on the vertices $u$ and $v$.
\end{itemize}
\end{definition}

The following well-known results on cones and suspensions (see, e.g., \cite[Theorems 8.2 and 25.4]{munkres_elements_1984})
will be very useful in the sequel.

\begin{proposition}\label{homConeSusp}
\begin{enumerate}
 \item\label{homCone}
 $\tilde H_i(v*\Delta)=0,\ \forall\, i\geq 0$.
 \item \label{homSusp}
 $\tilde H_0(\Sigma_u^v\Delta)=0$
 and
 $\tilde H_i(\Sigma_u^v\Delta)\simeq\tilde H_{i-1}(\Delta),\ \forall\, i\geq 1$.
\end{enumerate}
\end{proposition}

\smallskip

Associated to a graph $G$, one has its {\it independence complex} $\Delta(G)$ which is defined as the simplicial complex on the vertex set $V(G)$ such that $F\subset V(G)$ is a face of $\Delta(G)$ if and only if no edge of $G$ is a subset of $F$.
Observe that if $G$ is a graph and $W$ is an arbitrary subset of $V(G)$, one has that
$$\Delta(G)[W]=\Delta(G[W])\,.$$

\begin{remark}\label{flagcx}
A {\it flag complex} is a simplicial complex $\dd$ such that, for any $\sigma\subset V(\dd)$, if every pair of elements in $\sigma$ is a face of $\Delta$ then $\sigma$ is also a face of $\dd$. In particular, a flag complex containing all pairs of vertices is necessarily a simplex. Moreover, the independence complex $\Delta(G)$ of a graph $G$ is always, by definition, a flag complex.
\end{remark}

\begin{definition}\label{defDelLinkStar}
If $\Delta:=\Delta(G)$ is the independence complex of a graph $G$, for all $u\in V:=V(G)$, we consider three induced subcomplexes of $\Delta$ that will be featured in this paper.
Note that each of them is the independence complex of an induced subgraph of $G$:
\begin{itemize}
\item $\del_\Delta(u):=\{\sigma\in\Delta : u\not\in\sigma\}=\Delta[V\setminus\{u\}]$;
\item $\link_\Delta(u):=\{\sigma\in\Delta : u\not\in\sigma \hbox{ and } \sigma\cup \{u\}\in \Delta\}=\Delta[V\setminus(N_{G}(u) \cup \{u\})]$;
\item $\Star_\Delta(u):=\{\sigma\in\Delta : \sigma\cup \{u\}\in \Delta\}=\Delta[V\setminus N_{G}(u)]$.
\end{itemize}
\end{definition}

For any vertex $v\in V(G)$,
$\Star_\Delta(v)$ is a cone with appex $v$ and hence it is acyclic by Proposition~\ref{homConeSusp}.\ref{homCone}.
Since
$\del_\Delta(v)\cup \Star_\Delta(v)=\Delta$ and
$\del_\Delta(v)\cap \Star_\Delta(v)=\link_\Delta(v)$, we can apply (\ref{MVgeneral})
whenever $\link_\Delta(v)\not=\emptyset$ and get
{\small
\begin{multline}\label{MV}
 { \cdots\lrarrow \tilde H_{i}(\link_\Delta(v))\lrarrow
 \tilde H_{i}(\del_\Delta(v))
 \lrarrow \tilde H_{i}(\Delta)\lrarrow \tilde H_{i-1}(\link_\Delta(v))\lrarrow \cdots} \\
 {\cdots \lrarrow \tilde H_{0}(\link_\Delta(v))\lrarrow
 \tilde H_{0}(\del_\Delta(v))
 \lrarrow \tilde H_{0}(\Delta)\lrarrow 0}\,.
 \end{multline}
}

\smallskip
Let's focus now on bipartite graphs. Recall that a graph $G$ is {\it bipartite} if its vertex set can be splitted into two disjoint sets, $V(G)=X\sqcup Y$, in such a way that any edge of $G$ has one vertex in $X$ and the other in $Y$.
When one deals with bipartite graphs, it is usually convenient to use different symbols for variables in $X$ and variables in $Y$. We will denote variables in $X$ by $x_1,\ldots,x_n$ and variables in $Y$ by $y_1,\ldots,y_m$.
The {\it biadjacency matrix} of the bipartite graph $G$, $M(G)=(a_{i,j})\in\M01$, is defined by
$a_{i,j}=1$ if $\{x_i,y_j\}\in E(G)$, $0$ otherwise.
The {\it bipartite complement} of a bipartite graph $G$ is the bipartite graph $G^{bc}$ on the same vertex set as $G$, $V(G^{bc})=X\sqcup Y$, with $E(G^{bc})=\{\{x,y\}: x\in X, y\in Y, \{x,y\}\not\in E(G)\}$.
One has that $M(G^{bc})={\bf 1}_{n\times m}-M(G)$ where ${\bf 1}_{n\times m}$ is the $n\times m$ matrix whose entries are all 1.
Note that the bipartition $V(G)=X\sqcup Y$
may not be unique if the graph $G$ is not connected and the notions of biadjacency matrix or bipartite complement depend on the bipartition. That's the reason why in section~\ref{reg3} we will restrict ourselves to connected bipartite graphs.

\smallskip
The next lemma will be useful to handle the homology of the independence complex of a bipartite graph $G$. The last three items are rules that one can apply for a reduction to a simpler case by removing vertices of $G$ when the biadjacency matrix $M$ of $G$ satisfies some properties.

\begin{lemma}\label{lanota}
Let $G$ be a bipartite graph with bipartition $V(G)=\{x_1,\ldots,x_n\}\sqcup \{y_1,\ldots,y_m\}$,
biadjacency matrix $M=(a_{i,j})\in\M01$, and independence complex $\Delta$.
\begin{enumerate}
\item \label{launo}
If $M$ has a row or a column whose entries are all 0, then $\tilde{H}_i(\Delta)=0$ for all $i\geq 0$.
\item \label{lados}
If there exist $r$ and $c$ such that $a_{r,c}=1$ and the rest of entries on the row $r$ and the column $c$ are zeros then, for all $i>0$,
$$\tilde{H}_i(\Delta)\simeq\tilde{H}_{i-1}(\Delta[V(G)\setminus\{x_r,y_c\}])\,.$$
\item \label{latres}
If $M$ has more than one row {\rm(}resp. column{\rm)} and if the entries on the row $r$ {\rm(}resp. column $c${\rm)} are all 1 then,
for all $i\geq 0$,
$$\tilde{H}_i(\Delta)\simeq\tilde{H}_{i}(\Delta[V(G)\setminus\{x_r\}])\
{\rm(}resp.\ \tilde{H}_i(\Delta)\simeq\tilde{H}_{i}(\Delta[V(G)\setminus\{y_c\}]){\rm)}
\,.$$
\item \label{lacuatro}
If $M$ has a row $r$ {\rm(}resp. column $c${\rm)} with a unique zero entry, say $a_{r,c}=0$, and if there is another zero entry on the column $c$ {\rm(}resp. row  $r${\rm)} then for all $i\geq 0$,
$$\tilde{H}_i(\Delta)\simeq\tilde{H}_{i}(\Delta[V(G)\setminus\{x_r\}])\
{\rm(}resp.\ \tilde{H}_i(\Delta)\simeq\tilde{H}_{i}(\Delta[V(G)\setminus\{y_c\}]){\rm)}
\,.$$
\item \label{lacinco}
If $M$ has two rows $r$ and $r'$ {\rm(}resp. two columns $c$ and $c'${\rm)} such that
$\{j:\,a_{r,j}=0\}\subset \{j:\,a_{r',j}=0\}$
{\rm(}resp. $\{i:\,a_{i,c}=0\}\subset \{j:\,a_{i,c'}=0\}${\rm)}
then for all $i\geq 0$,
$$\tilde{H}_i(\Delta)\simeq\tilde{H}_{i}(\Delta[V(G)\setminus\{x_r\}])\
{\rm(}resp.\ \tilde{H}_i(\Delta)\simeq\tilde{H}_{i}(\Delta[V(G)\setminus\{y_c\}]){\rm)}
\,.$$
\end{enumerate}
\end{lemma}

\begin{proof}
\ref{launo}:
The vertex $z$ of $G$ corresponding to the row or column of $M$ with zero entries is isolated in $G$ and hence \delta\ is
a cone with appex $z$, so it is acyclic by Proposition~\ref{homConeSusp}.\ref{homCone}.

\ref{lados}:
By \ref{launo}, one has that $\tilde{H}_i(\del_\Delta(y_c))=\tilde{H}_i(\del_\Delta(x_r))=0$. Since
$\Delta=\del_{\Delta}(x_{r})\cup \del_\Delta(y_c)$ and $\del_{\Delta}(x_{r})\cap \del_\Delta(y_c)=\del_{\del_{\Delta}(x_{r})}(y_c)$,
the result follows from the Mayer-Vietoris sequence (\ref{MV}).

\ref{latres} and \ref{lacuatro} are particular cases of \ref{lacinco} that follows by applying again the Mayer-Vietoris sequence (\ref{MV}),
observing that $\link_\Delta(x_{r})$ is acyclic by \ref{launo}.
\end{proof}

\begin{example}\label{sigmam}
If $\Sigma_m$ is the independence complex of $mK_2$, then
$$
\dim_{ \K}(\tilde H_{i}(\Sigma_{m}))=\left\{\begin{array}{ll} 1 &\hbox{ if } i=m-1, \\ 0 &\hbox{ otherwise } . \end{array} \right.
$$
This can be shown by induction on $m\geq 1$ as follows. Since $\Sigma_{1}$ consists of two disjoint vertices, $\tilde H_{0}(\Sigma_{1})\simeq \K$ and $\tilde H_{i}(\Sigma_{1})=0$ for all $i>0$. If $m>1$, $\tilde H_{0}(\Sigma_{m})=0$ because
$\Sigma_{m}$ is connected and for $i>0$, applying Lemma~\ref{lanota}.\ref{lados},
one gets that $\tilde H_{i}(\Sigma_{m})\simeq\tilde H_{i-1}(\Sigma_{m-1})$ and the result follows.
\end{example}

\subsection{Some properties of the graded Betti numbers of an edge ideal}

Given a minimal graded free resolution of a homogeneous ideal $I$ in $R=\K[x_1,\cdots,x_n]$,
 $$
0\lar \bigoplus_jR(-j)^{\beta_{p,j}}\lar\ldots\lar
\bigoplus_jR(-j)^{\beta_{1,j}}\morfismo{\varphi_1}
 \bigoplus_jR(-j)^{\beta_{0,j}}\morfismo{\varphi_0}I\lar0\,,
 $$
the {\it regularity} of $I$ (in the sense of Castelnuovo-Mumford) is defined in terms of the graded Betti numbers $\beta_{i,j}$ as ${\reg(I)}:=\max\{j-i :\beta_{i,j}\not=0\}$. The graded Betti numbers of $I$ are usually arranged
in a table called the {\it Betti diagram} of $I$
where $\beta_{i,j}$ is placed in the $i$-th column and $(j-i)$-th row of the table. Note that the index of the last nonzero row in the Betti diagram of $I$ is its regularity.

\smallskip
If $I$ is a monomial ideal, we can provide the polynomial ring $R$ with the usual $\mathbb N^{n}$-multigrading and $I$ has a minimal multigraded free resolution. We can then define its {\it multigraded Betti numbers}, $\beta_{i,\mathbf m}$, as the number of minimal generators of degree ${\mathbf m}\in \mathbb N^{n}$ in the $i$-th syzygy module.

\smallskip
There is a one-to-one correspondence between squarefree monomial ideals generated in degree 2 and simple graphs:
associated to a simple graph $G$, one has the edge ideal $I(G)$ generated by the monomials of the form $x_ix_j$ with $\{x_i,x_j\}\in E(G)$.
The edge ideal $I(G)$ is indeed the Stanley-Reisner ideal of the independence complex $\Delta(G)$ of $G$, $I(G)=I_{\Delta(G)}$. The multigraded Betti numbers of an edge ideal $I(G)$ can be expressed in terms of the reduced homology of $\Delta(G)$ using Hochster's Formula (\cite[Theorem (5.1)]{hochster_cohen-macaulay_1977}) that we recall now.
For any $\mathbf m\in {\mathbb N}^n$ and $i\geq 0$, one has that $\beta_{i,\mathbf m}(I(G))=0$ if the monomial $\xx{\mathbf m}:=x_1^{m_1}\cdots x_n^{m_n}$ is not squarefree, and
\begin{equation} \label{hochsterMult}
\beta_{i,\mathbf{m}}(I(G))=\dim_\K\tilde{H}_{|\mathbf{m}|-i-2}(\dd(G)[W])
\end{equation}
otherwise, where $W:=\{j\in [n] : m_{j}=1\}$. The graded Betti numbers of $I(G)$ are then given by the following formula:
\begin{equation} \label{hochster}
\beta_{i,j}(I(G))=\sum_{W\subset V(G),|W|=j}\dim_{\K}(\tilde H_{j-i-2}(\Delta(G)[W]))\, .
\end{equation}

Hochster's Formula is a powerful tool when one wants to get information on the Betti numbers of edge ideals. It can be used for example to prove the following property on the graded Betti numbers of an edge ideal $I(G)$:

\begin{theorem}\label{BettiDiagEdgeIdeal}
For any $i\geq 0$ and any $j\geq i+2$, if $\beta_{i,j}(I(G))=\beta_{i,j+1}(I(G))=0$ then $\beta_{i+1,j+2}(I(G))=0$.
\end{theorem}

\begin{proof}
Denote by $\Delta:=\Delta(G)$ the independence complex of $G$ and
assume that $\beta_{i+1,j+2}(I(G))\not=0$. By (\ref{hochsterMult}), there exists $W\subset V(G)$ with $|W|=j+2$ such that $\dim_\K\tilde H_{j-i-1}(\dd[W])>0$. As $\dd[W]=\dd(G[W])$ is a flag complex, there exist $u,v\in W$ such that $\{u,v\}\not\in\dd(G)[W]$ by Remark~\ref{flagcx}. Consider then the following decomposition of $\Delta[W]$,
$$
\Delta[W]=\Delta[W\setminus\{u\}]\cup \Delta[W\setminus\{v\}]\,,
$$
 with
$\Delta[W\setminus\{u\}]\cap \Delta[W\setminus\{v\}]=\Delta[W\setminus\{u,v\}]$ which is not empty since $|W|=j+2>2$.
Invoking Hochster's Formula (\ref{hochsterMult}) again, one has that
$\tilde H_{j-i-1}(\Delta[W\setminus\{u\}])=\tilde H_{j-i-1}(\Delta[W\setminus\{v\}])=0$ because $\beta_{i,j+1}(I)=0$
and
$
\tilde H_{j-i-2}(\Delta[W\setminus\{u,v\}])=0
$
because $\beta_{i,j}(I)=0$.
Using the Mayer-Vietoris sequence (\ref{MVgeneral}), one gets that
$\tilde H_{j-i-1}(\Delta[W])=0$, a contradiction.
\end{proof}

Note that Theorem~\ref{BettiDiagEdgeIdeal} can easily
be extended to monomial ideals generated in degree two that may not be squarefree through polarization (see section \ref{nonsqf}). The following direct consequence answers
a question by Aldo Conca
who asked if our \cite[Theorem 5.2]{fernandez-ramos_first_2009} could be improved in this direction. In \cite{conca_2010}, Avramov, Conca and Iyengar proved bounds for the syzygies of Koszul algebras and that question arose in this context.

\begin{corollary}\label{conca}
Let $I$ be a monomial ideal generated in degree two and denote by $u_i$ the maximal degree of a minimal generator in its $i$-th syzygy module. Then, for all $i\geq 0$, $u_{i+1}\leq u_i +2$.
\end{corollary}

\medskip
When  an edge ideal $I(G)$ has a linear resolution, all the nonzero entries in its Betti diagram are located on the first row.
Fr\"oberg proved that an edge ideal $I(G)$ has a linear resolution if and only if the graph $G^{c}$ is chordal. We can rephrase this nice combinatorial characterization as follows:

\begin{theorem}[{\cite{froberg_stanley-reisner_1990}}] \label{froberg}
An edge ideal $I(G)$ has regularity 2 if and only if $G^{c}$ does not have induced cycles.
\end{theorem}

In \cite{eisenbud_restricting_2005}, the authors go one step further and show that if $\reg(I(G))>2$,
the non-linear syzygies appear for the first time at the $(t-3)$-th step of the resolution where
$t$ is the minimal length of an induced cycle in $G^c$. This result is contained in the following stronger statement:

\begin{theorem}[{\cite{fernandez-ramos_first_2009}}]\label{nuestro}
If $I(G)$ is an edge ideal with $\reg(I(G))>2$, let $t\geq 4$ be the minimal length of an induced cycle in $G^{c}$. Then:
\begin{itemize}
\item $\beta_{i,j}(I(G))=0$ for all $i<t-3$ and $j> i+2$;
\item $\beta_{t-3,t}(I(G))=|\{$induced $t$-cycles in $G^{c}\}|$;
\item  $\beta_{t-3,j}(I(G))=0$ for all $j>t$;
\item
for any $\mathbf m\in {\mathbb N}^n$ such that $\vert{\mathbf m}\vert=t$, one has that
$\beta_{t-3,\mathbf m}(I(G))=1$ if $\mathbf m\in \{0,1\}^{n}$ and $G^{c}[W]\simeq C_{t}$ where $W:=\{x_i\,:\ m_i=1\}$. Otherwise, $\beta_{t-3,\mathbf m}(I(G))=0$.
\end{itemize}
\end{theorem}

Observe that in the previous theorem, induced cycles in $G^c$ play an important role. That's why we previously
focused on a particular family of edge ideals, those associated to complements of cycles, and gave in \cite[Proposition 3.1]{fernandez-ramos_first_2009}
closed combinatorial formulas for all its graded Betti numbers.

\medskip

Following the same philosophy, we will focus now on graphs that are the bipartite complement of an even cycle since induced even cycles
in the bipartite complement of an arbitrary graph will play a fundamental role later in our main Theorem~\ref{main1}.
Again, for this family of graphs, we describe all the graded Betti numbers of the associated edge ideals
in Theorem~\ref{bettibcciclo}.

\section{Bipartite complement of a cycle of even length}\label{cycle}

The following result is a direct consequence of Propositions \ref{ThirdRowBettiCBC},
\ref{FirstRowBettiCBC} and \ref{SecRowBettiCBC}, that we will prove in this section.

\begin{theorem}\label{bettibcciclo}
The edge ideal associated to the bipartite complement of an even cycle
of length $t:=2s\geq 6$ has regularity 4 and its Betti diagram is
\begin{center}
\rm{
\begin{tabular}{r|c|ccc|cc|c|c|}
                        \multicolumn{1}{c}{}  & \multicolumn{1}{c}{0}              & \multicolumn{1}{c}{1}             & \ldots & \multicolumn{1}{c}{$s-3$}  & \ldots & \multicolumn{1}{c}{$t-5$} & \multicolumn{1}{c}{$t-4$} \\
 \cline{2-5}            \multicolumn{1}{c}{2} & \multicolumn{1}{>{\columncolor[rgb]{0.82,0.82,0.82}} c}{$\beta_{0,2}$} & \multicolumn{1}{>{\columncolor[rgb]{0.82,0.82,0.82}} c}{$\beta_{1,3}$} & \multicolumn{1}{>{\columncolor[rgb]{0.82,0.82,0.82}} c}{\ldots} & \multicolumn{1}{>{\columncolor[rgb]{0.82,0.82,0.82}} c|}{$\beta_{s-3,s-1}$}          &        &                       \multicolumn{2}{c}{}            \\
 \cline{2-2} \cline{6-7}\multicolumn{1}{c}{3} &                                    & \multicolumn{1}{>{\columncolor[rgb]{0.82,0.82,0.82}} c}{$\beta_{1,4}$}                    & \multicolumn{1}{>{\columncolor[rgb]{0.82,0.82,0.82}} c}{\ldots}& \multicolumn{1}{>{\columncolor[rgb]{0.82,0.82,0.82}} c}{\ldots} & \multicolumn{1}{>{\columncolor[rgb]{0.82,0.82,0.82}} c}{\ldots}& \multicolumn{1}{>{\columncolor[rgb]{0.82,0.82,0.82}} c|}{$\beta_{t-5,t-2}$}         & \multicolumn{1}{c}{}      \\
 \cline{3-8}            \multicolumn{1}{c}{4} & \multicolumn{6}{c|}{}                                                                                                                             &          \multicolumn{1}{>{\columncolor[rgb]{0.82,0.82,0.82}} c|}{1}                \\
\cline{8-8}                                                                       \multicolumn{8}{c}{}
\end{tabular}
}
\end{center}
where the nonzero entries are located in the shadowed area. Moreover,
$\beta_{j-2,j}$ for $2\leq j\leq s-1$,
and $\beta_{j-3,j}$ for $4\leq j\leq t-2$, are given respectively
by the following closed combinatorial formulas:
\begin{eqnarray*}
\beta_{j-2,j}&=&\sum_{k=1}^{j-1}\sum_{c=1}^{k}\frac{s}{c}\binom{k-1}{c-1}\binom{s-k-1}{c-1}\binom{s-k-c}{j-k},\\
\beta_{j-3,j}&=&\sum_{m=2}^{\lfloor j/2\rfloor}t\frac{m-1}{m}
\binom{t-j-1}{m-1}
\sum_{a=0}^{j-2m}
\binom{j-m-a-1}{m-1}\binom{t-j-m}{a}.
\end{eqnarray*}
\end{theorem}

\medskip
Let $G:=C^{bc}_{2s}$ be the bipartite complement of an even cycle $C_{2s}$ with at least 6 vertices, i.e., $s\geq 3$.
The vertices and edges of the even cycle $C_{2s}$ will be $V:=\{x_1,\ldots,x_s\}\sqcup \{y_1,\ldots,y_s\}$
and $\{\{x_1,y_1\},\{y_1,x_2\},\ldots,\{y_s,x_1\}\}$ respectively along this section. We will
sometimes refer to the two subsets in the bipartition of $V$ as $X$ and $Y$.
The biadjacency matrix $M$ of $G$ has exactly two zero entries
on each row and column:
$$
M=
\left( \footnotesize{
\begin{array}{cccccccc}
 0        &       1   &     \ldots &  \ldots   & 1      &0\\
 0        &       0   &     \ddots &           & \vdots &1\\
 1        &       0   &        0   &  \ddots   & \vdots &\vdots \\
\vdots    &   \ddots  &     \ddots &  \ddots   & 1      &\vdots \\
\vdots    &           &     \ddots &      0    & 0      &1\\
 1        &   \ldots  &     \ldots &      1    & 0      &0\\
\end{array}
}\right).
$$
In order to use Hochster's Formula to determine the graded Betti numbers of $I(G)$, we need to compute the reduced simplicial homologies $\tilde H_{i}(\Delta(G[W]))$ for all subsets $W$ of $V$. The case $W=V$ is solved in Proposition \ref{homDelta}. Its proof requires the following lemma.

\begin{lemma}
For every $v\in V$, $\del_{\Delta(G)}(v)$ is acyclic.
\end{lemma}

\begin{proof}
Without loss of generality, let's choose $v=x_1$. As observed in Definition~\ref{defDelLinkStar},
$\del_{\Delta(G)}(x_1)$ is the independence complex of $G[V\setminus\{x_1\}]$
whose biadjacency matrix $N$ is obtained by removing the first row of $M$. Observe that the first and last columns of $N$
satisfy the condition in Lemma~\ref{lanota}.\ref{lacuatro} and hence can be removed. Again, the first and last rows of this
new matrix satisfy the same condition and we remove them. Recursively,
when $s$ is odd (respectively even), we reduce the computation of the homology to the case of the independence complex of
a graph whose biadjacency matrix is a
$2\times 3$ (respectively $3\times 2$) matrix whose central column (respectively row) has its two entries equal to zero and
$\del_{\Delta(G)}(v)$ is acyclic by Lemma~\ref{lanota}.\ref{launo}.
\end{proof}

\begin{proposition}\label{homDelta}
$\displaystyle{
\dim_{ \K}(\tilde H_{i}(\Delta(G)))=\left\{\begin{array}{ll}  1 &\hbox{ if }  i=2, \\ 0 & \hbox{ otherwise } .\end{array} \right.
}$
\end{proposition}

\begin{proof}
Since $\Delta(G)$ is connected, $\dim_{ \K}(\tilde H_{0}(\Delta(G)))=0$. For $i>0$, using
the Mayer-Vietoris sequence (\ref{MV}), the previous lemma implies that
$\tilde H_i(\dd(G))\simeq \tilde H_{i-1}(\link_{\dd(G)}(v))$
for every vertex $v$ of $G$ and $i>0$. The biadjacency matrix (or its transpose)
of $G[V\setminus(N_G(v)\cup\{v\})]$ is a $2\times (s-1)$ matrix with exactly two zero entries, one in each row and located on two different columns.
Applying Lemma~\ref{lanota}.\ref{latres} (if $s>3$) as many times as necessary, one gets that
$\tilde H_i(\dd(G))\simeq \tilde H_{i-1}(\Sigma_2)$ and the result follows from Example~\ref{sigmam}.
\end{proof}

Let $W$ be now a proper subset of $V$, $W\subsetneq V=X\sqcup Y$. Set $W_{X}:=W\cap X$, $W_{Y}:=W\cap Y$, and
denote by $k_W$ the number of connected components of the graph $C_{2s}[W]$ that are not isolated vertices.
Note that if $k_W\neq 0$, then $W_{X}\neq\emptyset$ and $W_{Y}\neq\emptyset$.

\begin{proposition}\label{componentes}
  \begin{enumerate}
\item \label{emptyacyclic}
If $W_{X}=\emptyset $ or $W_{Y}=\emptyset$ then $\Delta(G[W])$ is acyclic.
\item \label{nonemptyKzero}
If $W_{X}\neq\emptyset $, $W_{Y}\neq\emptyset$ and $k_W=0$
then
$$
  \dim_{ \K}(\tilde H_{i}(\Delta(G[W])))=\left\{\begin{array}{ll}  1 &\hbox{ if }  i=0, \\ 0 & \hbox{ otherwise } .\end{array} \right.
$$
\item \label{nonemptyKnonzero}
If $k_W>0$ then
$$
 \dim_{ \K}(\tilde H_{i}(\Delta(G[W])))=\left\{\begin{array}{ll}  k_W-1 &\hbox{ if }  i=1, \\ 0 & \hbox{ otherwise } . \end{array} \right.
$$
 \end{enumerate}
\end{proposition}

\begin{remark}\label{rmk}
Observe that
$G^{c}=C_{2s}\cup K_{X} \cup K_{Y}$,
where $K_A$ denotes the complete graph on a set of vertices $A\subset V$.
Thus,
$G^{c}[W]=C_{2s}[W]\cup K_{W_X} \cup K_{W_Y}$.
Since $K_{W_X}$ and $K_{W_Y}$ are connected, if one of them is empty or if they are connected to each other in $G^{c}[W]$, i.e.,
if $k_W\neq 0$, then $G^{c}[W]$ is connected. Otherwise, $K_{W_X}$ and $K_{W_Y}$ are its connected components.
Thus, the condition in Proposition~\ref{componentes}.\ref{nonemptyKzero} is satisfied if and only if
$G^{c}[W]$ is not connected. When $W_X\not=\emptyset$ and $W_Y\not=\emptyset$, denote by $M[W]$ the biadjacency matrix of $G[W]$. It is easy to check that if $W_X\neq\emptyset$ and $W_Y\neq\emptyset$, then
\begin{center}
$k_W=0$ $\Leftrightarrow$ $N_{C_{2s}}(W_X)\cap W_Y=\emptyset$ $\Leftrightarrow$ $M[W]$ has no zero entries.
\end{center}
\end{remark}

\begin{proof}[{Proof of Proposition~\ref{componentes}}]
If $W_{X}=\emptyset $ or $W_{Y}=\emptyset$, $\Delta(G[W])$ is a simplex and \ref{emptyacyclic} follows.
Assume now that
$W_{X}\neq\emptyset $, $W_{Y}\neq\emptyset$ and $k_W=0$. Then, $M[W]$ has no zero entries by Remark~\ref{rmk} and,
by Lemma~\ref{lanota}.\ref{latres}, for all $i\geq 0$,
$\tilde H_{i}(\Delta(G[W]))\simeq\tilde H_{i}(\Delta(K_2))=\tilde H_{i}(\Sigma_1)$. So \ref{nonemptyKzero} follows from Example~\ref{sigmam}.

\smallskip
Assume now that $M[W]$ has at least one zero entry.
First observe that the number of zero entries
in any row and column of $M[W]$
is at most two (and that for at least two of the columns or rows, it is one).
By Lemma~\ref{lanota}.\ref{latres}, the dimension of the reduced homologies will not change if we remove from $M[W]$ any row and any column with no zero entry. In other words, since a row or a column of $M[W]$ with no zero entry corresponds to an isolated vertex in $C_{2s}[W]$, if $W'$ is the subset of $W$ formed by all the elements in $W$ that
are not isolated vertices in $C_{2s}[W]$, one has that
$\tilde H_{i}(\Delta(G[W]))\simeq\tilde H_{i}(\Delta(G[W']))$.
Moreover, $k_W=k_{W'}$. Using now Lemma~\ref{lanota}.\ref{lacuatro}, one can remove from $M[W']$ any row (resp. column) with exactly one zero entry and such that
in the column (resp. row) where this zero entry is located, there is another zero entry. Such a row or column of $M[W']$ corresponds to
a vertex of $C_{2s}[W']$ of degree one whose (unique) neighbor is of degree two. Removing such a vertex does not change the number of
connected components of $C_{2s}[W']$ and it creates in $C_{2s}[W']$ a vertex of the same kind, until we reach a vertex
of degree one whose neighbor also has degree one. Thus, $\tilde H_{i}(\Delta(G[W]))\simeq \tilde H_{i}(\Delta((k_W K_2)^{bc}))$ for all $i\geq 0$. The result is now a direct consequence of the following technical lemma.
\end{proof}

\begin{lemma}\label{lemmaIm}
For all $m\geq 1$,
$\displaystyle{
\dim_{\K}(\tilde H_{i}(\Delta((mK_2)^{bc})))=\left\{\begin{array}{ll} m-1 &\hbox{ if } i=1, \\ 0 & \hbox{ otherwise.} \end{array} \right.
}$
\end{lemma}

\begin{proof}
The biadjacency matrix of the graph $(mK_2)^{bc}$ is an $m\times m$ matrix whose entries are all 1 except the
ones on the principal diagonal that are zero. Denote by $\Theta_m$ its independence complex,
$\Theta_m:=\Delta((mK_2)^{bc})$.
Since $\Theta_m$ is connected,
$\dim_{\K}(\tilde H_{0}(\Theta_m))=0$ for all $m\geq 1$.

\smallskip
In order to determine the homology for $i\geq 1$, consider the family of subcomplexes of $\Theta_m$,
$F=\{\Theta_m[X],\Theta_m[Y],\{x_1,y_1\},\ldots,\{x_m,y_m\}\}$ whose elements we index by
$x,y,z_1,\ldots,z_m$. Recall that the {\it nerve} of $F$, $N(F)$, is the simplicial complex on the
vertex set $V_F:=\{x,y,z_1,\ldots,z_m\}$ whose faces are the subsets of $V_F$ such that the intersection of the
corresponding elements in $F$ is non empty.
The simplicial complex $N(F)$ has $2m$ facets, $\{x,z_i\}$ and $\{y,z_i\}$ for all $i=1,\ldots,m$.
Since $\Theta_m=\bigcup_{i\in V_F}F_i$,
applying the Nerve Theorem (see, for example, \cite[Theorem 10.6]{bjorner_topological_1995}), one gets that
$\tilde H_{i}(\Theta_m)\simeq \tilde H_{i}(N(F))$ for all $i\geq 0$.
On the other hand, $N(F)=\displaystyle{\Sigma_{x}^{y}\langle \{z_{1}\},\ldots,\{z_{m}\}\rangle}$
and hence, by Proposition \ref{homConeSusp}.\ref{homSusp},
$\tilde H_{i}(N(F))\simeq \tilde H_{i-1}(\langle \{z_{1}\},\ldots,\{z_{m}\}\rangle)$ for all $i\geq 1$
and the result follows.
\end{proof}

As a straightforward consequence, one gets that the last row of the Betti diagram of $I(C_{2s}^{bc})$ with a nonzero
entry is the one indexed by 4, i.e., $\reg{I(C_{2s}^{bc})}=4$, and that $\beta_{2s-4,2s}(I(C_{2s}^{bc}))=1$
is the only nonzero entry on this row.

\begin{proposition}\label{ThirdRowBettiCBC}
\begin{itemize}
\item $\beta_{i,j}(I(C_{2s}^{bc}))=0$ if $j>i+4$;
 \item $
\beta_{i,i+4}(I(C_{2s}^{bc}))=\left\{\begin{array}{ll} 1 & \hbox{ if }  i=2s-4\,, \\ 0 &\hbox{ otherwise }\,. \end{array} \right.
$
\end{itemize}
\end{proposition}

\begin{proof}
Putting together Propositions \ref{homDelta} and \ref{componentes}, one has that for every subset $W$ of $V$,
$\dim_{ \K}(\tilde H_{i}(\Delta(G[W])))=0$ for all $i\notin\{0,1,2\}$. Moreover,
$\dim_{ \K}(\tilde H_{2}(\Delta(G[W])))\neq 0$ if and only if $W=V$ and
$\dim_{ \K}(\tilde H_{2}(\Delta(G)))=1$. The result then follows from Hochster's Formula (\ref{hochster}).
\end{proof}

\medskip

In order to complete the description of the Betti diagram of $I(C_{2s}^{bc})$, one has to determine
the graded Betti numbers on the first two rows, i.e., $\beta_{i,j}$ for $i+2\leq j\leq i+3$.

\medskip

We start with the first row. Using Hochster's Formula (\ref{hochster}) and
Proposition~\ref{componentes}.\ref{nonemptyKzero}, one needs to determine all the proper subsets $W$ of $V$ such that
$W_X\neq\emptyset$, $W_Y\neq\emptyset$ and $G^{c}[W]$ is not connected.
Indeed, $\beta_{i,i+2}(I(G))$ is the number of induced subgraphs $G^{c}[W]$ on $i+2$ vertices
that are non connected.

\medskip
Let's denote by $C_X$ the cycle on the vertex set $X$ whose edges are $\{x_1,x_2\}$, $\{x_2,x_3\}$, \ldots, $\{x_s,x_1\}$. Note that
the edges of $C_X$ correspond to the pairs $\{x_i,x_j\}$ of elements in $X$ such that
$N_{C_{2s}}(x_{i})\cap N_{C_{2s}}(x_{j})\not=\emptyset$.

\begin{lemma}\label{LemmaNonConnected}
Assume that $G^{c}[W]$ is not connected.
\begin{enumerate}
\item\label{InLemmaNonConn1}
There exists
$x\in W_{X}$ such that $N_{C_{2s}}(x)\not \subset N_{C_{2s}}(W_{X}\backslash \{x\})$.
\item\label{InLemmaNonConn2}
 $|N_{C_{2s}}(W_X)|=|W_X|+|\comp(C_X[W_X])|$ where $\comp(C_X[W_X])$ is the set of connected components of
 $C_X[W_X]$.
\end{enumerate}
\end{lemma}

\begin{proof}
If $N_{C_{2s}}(x) \subset N_{C_{2s}}(W_X\backslash \{x\})$ for some $x\in W_{X}$, then $N_{C_{X}}(x) \subset W_X$. Thus, if
$N_{C_{2s}}(x) \subset N_{C_{2s}}(W_X\backslash \{x\})$ for all $x\in W_{X}$, then $W_X=X$
and one can not have $W_Y\neq\emptyset$ and $N_{C_{2s}}(W_X)\cap W_Y=\emptyset$. This implies that
$G^{c}[W]$ is connected by Remark~\ref{rmk} and \ref{InLemmaNonConn1} follows.

\smallskip
We will prove \ref{InLemmaNonConn2} by induction on $r:=|W_{X}|$.
If $W_{X}=\{x\}$ then $|N_{C_{2s}}(x)|=2$, $|\{x\}|=1$, $|\comp((C_{x})[\{x\}])|=1$ and the statement holds.
Consider now $W$ such that $|W_{X}|=r>1$ and assume that the statement holds for subsets $X'$ such that $|W_{X'}|=r-1$.
By \ref{InLemmaNonConn1}, we know that there exists $x_0\in W_X$ such that $N_{C_{2s}}(x_0)\not \subset N_{C_{2s}}(W_{X}\backslash \{x_0\})$,
and one has two possibilities:
\begin{itemize}
 \item
If $x_0$ is connected in $C_{X}$ to some $x\in W_{X}\backslash\{x_0\}$, i.e., if
$N_{C_{2s}}(x_{0})\cap  N_{C_{2s}}(W_{X}\backslash \{x_{0}\})\not=\emptyset$, then
$|N_{C_{2s}}(W_X)|=|N_{C_{2s}}(W_X\backslash \{x_{0}\})|+1$. In this case,
$|\comp((C_{X})[W_{X}])|=|\comp(C_{X}[W_{X}\backslash\{x_{0}\}])|$;
\item
Otherwise, $|N_{C_{2s}}(W_X)|=|N_{C_{2s}}(W_X\backslash \{x_{0}\})|+2$ and
$|\comp(C_{X}[W_{X}])|=|\comp(C_{X}[W_{X}\backslash\{x_{0}\}])|+1$.
\end{itemize}
In both cases, applying our inductive hypothesis, one gets that
$|N_{C_{2s}}(W_X)|=|W_X\backslash \{x_{0}\})|+|\comp(C_{X}[W_{X}])|+1=|W_X|+|\comp(C_{X}[W_{X}])|$.
\end{proof}

The nonzero entries on the first row of the Betti diagram are given by the following result.

\begin{proposition}\label{FirstRowBettiCBC}
\begin{enumerate}
\item\label{FirstRowBettiCBCzeros}
For all $j\geq s$, $\beta_{j-2,j}(I(C_{2s}^{bc}))=0$.
\item\label{FirstRowBettiCBCnonzero}
For $j=2,\ldots,s-1$,
$$
\beta_{j-2,j}(I(C_{2s}^{bc}))=\sum_{k=1}^{j-1}\sum_{c=1}^{k}\frac{s}{c}\binom{k-1}{c-1}\binom{s-k-1}{c-1}\binom{s-k-c}{j-k}\ .
$$
\end{enumerate}
\end{proposition}

\begin{proof}
Consider a proper subset $W$ of $V$ with $|W|=j\geq 2$. As already observed,
$\tilde H_{0}(\Delta(G[W]))$ will contribute (by 1) to $\beta_{j-2,j}(I(C_{2s}^{bc}))$ in Hochster's Formula if and only if
$G^c[W]$ is not connected.

\smallskip
By Remark~\ref{rmk}, if $G^{c}[W]$ is not connected then
$|W_X|>0$, $|W_Y|>0$ and $|N_{C_{2s}}(W_X)|+|W_Y|\leq|Y|=s$. Thus,
$|W_Y|\leq s-|N_{C_{2s}}(W_X)|<s-|W_X|$
by Lemma~\ref{LemmaNonConnected}.\ref{InLemmaNonConn2}
since $W_X\not=\emptyset$ and hence
$|\comp(C_{X}[W_{X}])|\neq 0$. It implies
that $0<|W_X|<|W|=|W_X|+|W_Y|<s$. Thus if $|W|\geq s$, $G^{c}[W]$ is connected and \ref{FirstRowBettiCBCzeros} follows.

\smallskip
Now for $j$ with $2\leq j\leq s-1$,
we have to count how many subsets $W$ of $V$ with $|W|=j$ satisfy that $G^c[W]$ is not connected.
For each choice of $W_{X}$ with $k$ elements ($1\leq k\leq j-1$ in order to have $W_{X}\not= \emptyset$
and $W_{Y}\not= \emptyset$), we must choose $j-k$ elements from $Y\backslash N_{C_{2s}}(W_X)$ for $W_{Y}$,
so there are $\binom{s-|N_{C_{2s}}(W_X)|}{j-k}=\binom{s-k-|{\small\comp}(C_X[W_{X}])|}{j-k}$ possible choices
by Lemma~\ref{LemmaNonConnected}.\ref{InLemmaNonConn2}.
If we fix the number of connected components
of $C_X[W_{X}]$ and denote it by $c$, according to \cite[Lemma 3.3]{fernandez-ramos_first_2009},
there are $\frac{s}{c}\binom{k-1}{c-1}\binom{s-k-1}{c-1}$ possible subsets $W_{X}$ with $|W_{X}|=k$ and $|\comp(C_X[W_{X}])|=c$,
and the result follows.
\end{proof}

\begin{corollary} \label{beta02}
The first and the last nonzero entries on the first row of the Betti diagram of $I(C_{2s}^{bc})$ coincide, i.e.,
$\beta_{s-3,s-1}(I(C_{2s}^{bc}))=\beta_{0,2}(I(C_{2s}^{bc}))$.
\end{corollary}

\begin{proof}
For $j=s-1$ one has that
$\binom{s-k-c}{j-k}\not=0$ if and only if $c=1$. In this case $\binom{k-1}{c-1}=\binom{s-k-1}{c-1}=\binom{s-k-c}{j-k}=1$,
and hence $\beta_{s-3,s-1}(I(C_{2s}^{bc}))=\sum_{k=1}^{s-2}s=s(s-2)=|E(\bc C_{2s})|=\beta_{0,2}(I(C_{2s}^{bc}))$.
\end{proof}

The description of the Betti diagram of $I(C_{2s}^{bc})$ will be complete once we give the graded Betti numbers located on the second row.
This is our next result.

\begin{proposition}\label{SecRowBettiCBC}
\begin{enumerate}
\item\label{SecondRowBettiCBCzeros}
For all $j\geq 2s-1$, $\beta_{j-3,j}(I(C_{2s}^{bc}))=0$.
\item\label{SecondRowBettiCBCnonzero}
For $j=4,\ldots,2s-2$,
{\small
$$
\beta_{j-3,j}(I(C_{2s}^{bc}))=\sum_{m=2}^{\lfloor j/2\rfloor}(m-1)\sum_{a=0}^{j-2m}\frac{2s}{m}
\binom{j-m-a-1}{m-1}\binom{2s-j-1}{m-1}\binom{2s-j-m}{a}\ .
$$
}
\end{enumerate}
\end{proposition}

\begin{proof}
By Proposition \ref{componentes},
$\tilde H_{1}(\Delta(G[W]))$ will contribute to $\beta_{j-3,j}(I(C_{2s}^{bc}))$ in Hochster's Formula (\ref{hochster}) if and only if
$W$ is a proper subset of $V$ with $|W|=j\geq 4$ such that
$C_{2s}[W]$ has at least 2 connected components that are not isolated vertices. More precisely,
denoting by $w(j,m)$ the number of proper subsets $W$ of $V$ with $|W|=j$
and such that $C_{2s}[W]$ has $m$ connected components that are not isolated vertices, then
\begin{equation}\label{valorbetti}
 \beta_{j-3,j}(I(G))=\sum_{m=2}^{\lfloor\frac{j}{2}\rfloor}(m-1)w(j,m) \ .
\end{equation}
In particular, since for any subset $W$ of $V$ with $2s-1$ elements, one has that $C_{2s}[W]$ is connected,
\ref{SecondRowBettiCBCzeros} follows.

\smallskip
Now for $j\leq 2s-2$, denote by $W(j,m,a)$ the set of proper subsets $W$ of $V$ with $|W|=j$
and such that $C_{2s}[W]$ has $a$ isolated vertices and $m$ connected components that are not isolated vertices. Then, $w(j,m)=\sum_{a=0}^{j-2m}w(j,m,a)$ where
$w(j,m,a)=|W(j,m,a)|$, and we are reduced to compute $w(j,m,a)$
for all possible $j,m,a$.

\smallskip
As in the proof of \cite[Lemma 3.3]{fernandez-ramos_first_2009}, observe that a subset $W$ of $V$ can be represented
as a vector of length $2s$ whose $\ell$-th entry is 1 if the $\ell$-th element
in $V$ belongs to $W$, 0 otherwise.
Using this correspondence, the number of nonzero entries in this vector is the number of vertices in $C_{2s}[W]$
and the number of blocks of nonzero entries is related to the number of connected components of $C_{2s}[W]$.
In order to avoid distinguishing cases as when the vector starts/ends with 1/0, we
will allow to modify the starting vertex and focus only on vectors whose first
entry is 1 and last entry is 0. Denote by $B(2s,j,k)$ the set of vectors of length $2s$, with entries in $\{0,1\}$,
whose first entry is 1 and last entry is 0, and whose $j$ nonzero entries are located in $k$ different blocks.
Let $H(j,m,a)$ be the subset of $B(2s,j,m+a)$ formed by vectors with
$m$ blocks of 1's of length strictly bigger than 1 and $a$ blocks of 1's of length 1 and whose first block of nonzero entries has length strictly bigger than 1. To each element $w$ in $H(j,m,a)$ corresponds $2s$ elements in $W(j,m,a)$ (one for each choice of a vertex of $C_{2s}$ as the vertex corresponding to the first entry of $w$), and to an element $W$ in $W(j,m,a)$ corresponds $m$ distinct elements in $H(j,m,a)$ (one for each connected component of $C_{2s}[W]$ that we choose as the one that gives the first block of nonzero entries in the vector). Thus, $w(j,m,a)=\frac{2s}{m}|H(j,m,a)|$.

\smallskip
Finally, in order to determine $|H(j,m,a)|$, note that each element in $H(j,m,a)$ comes from a vector $h$ in $B(2s-m-a,j-m-a,m)$
by adding 1 in each block of 1's of $h$ (there are $m$), and by inserting $a$ times a 1 between two zero entries of $h$.
As already observed in the proof of \cite[Lemma 3.3]{fernandez-ramos_first_2009},
$|B(2s-m-a,j-m-a,m)|=\binom{j-m-a-1}{m-1}\binom{2s-j-1}{m-1}$. Moreover, for any element in $B(2s-m-a,j-m-a,m)$,
each block of zero entries of length $\ell$ will give $\ell-1$ places where one can add a 1 between two zero entries, and since and element in $B(2s-m-a,j-m-a,m)$ has $2s-j$ zero entries located in $m$ different blocks, each element in $B(2s-m-a,j-m-a,m)$ will provide
$\binom{2s-j-m}{a}$ elements in $H(j,m,a)$.
Putting all together, one gets that $w(j,m,a)=\frac{2s}{m}\binom{2s-j-m}{a}\binom{j-m-a-1}{m-1}\binom{2s-j-1}{m-1}$ and we are done.
\end{proof}

\begin{corollary}
The first and the last nonzero entries on the second row of the Betti diagram of $I(C_{2s}^{bc})$ coincide, i.e.,
$\beta_{2s-5,2s-2}(I(C_{2s}^{bc}))=\beta_{1,4}(I(C_{2s}^{bc}))$.
\end{corollary}

\begin{proof}
For $j=2s-2$, $\binom{2s-j-1}{m-1}\not=0$ if and only if $m=2$, and then $\binom{2s-j-m}{a}\not=0$ if and only if $a=0$,
and hence $\beta_{2s-5,2s-2}(I(C_{2s}^{bc}))=\frac{2s}{2}\binom{2s-5}{1}=s(2s-5)$.
On the other hand,
$\beta_{1,4}(I(C_{2s}^{bc}))=\frac{2s}{2}\binom{1}{1}\binom{2s-5}{1}\binom{2s-6}{0}$ and we are done.
\end{proof}

\begin{remark}
Recall from \cite{katzman_characteristic-independence_2006} that the induced matching number of a graph $G$ and the
Castelnuovo-Mumford regularity of the edge ideal $I(G)$ satisfy that $\mu(G)+1\leq \reg{(I(G))}$.
In the case of the bipartite complement of an even cycle, one can easily determine the induced matching number.
Since $C_6^{bc}= 3K_2$, $\mu(C^{bc}_6)=3$.
Now if $s>3$, $3K_2\not< C_{2s}^{bc}$ because $C_6\not<C_{2s}$, and hence $\mu(C_{2s}^{bc})<3$.
As $C_{2s}^{bc}[\{x_1,x_3,y_1,y_2\}]$ is formed by the
two non-connected edges $\{x_1,y_2\}$ and $\{y_1,x_3\}$, $\mu(C^{bc}_{2s})=2$ if $s>3$.
Hence, the matching number of $C^{bc}_{2s}$ and the Castelnuovo-Mumford regularity of
$I(C_{2s}^{bc})\subset R=\K[x_1,\ldots,x_s,y_1,\ldots,y_s]$ are related as follows:
$$
\reg{(I(C_{2s}^{bc}))}=
\left\{\begin{array}{ll}  \mu(C^{bc}_{2s})+1 & \hbox{ if }  s=3\,, \\ \mu(C^{bc}_{2s})+2 &\hbox{ if }  s\geq 4\,. \end{array} \right.
$$
This is not the only difference between
the cases $s=3$ and $s\geq 4$. Indeed, $R/I(C_{6}^{bc})$ is a complete intersection while, for $s\geq 4$, $R/I(C_{2s}^{bc})$ is not even Cohen-Macaulay (if it was then it would be Gorenstein which is impossible since its Betti diagram is not symmetric).
\end{remark}

\section{Regularity 3 in bipartite edge ideals}\label{reg3}

In this section we focus on edge ideals associated to bipartite graphs, which we call {\it bipartite edge ideals}.
We only consider connected graphs because the Betti numbers of the edge ideal associated to a disconnected graph can be computed from the Betti numbers of the edge ideals associated to its connected components; see \cite[Lemma 2.1]{jacques_betti_2005}.

\smallskip
Bipartite edge ideals
having regularity 2 can be characterized using Theorem~\ref{froberg}.
They are shown to be the edge ideals associated to Ferrer's graphs
in \cite[Theorem 4.2]{corso_monomial_2009}.

\smallskip
Our aim here is to prove our main results, Theorems~\ref{characterization} and \ref{main1}. The first one, analogous to Fr\"oberg's classical Theorem~\ref{froberg},
provides a combinatorial characterization of bipartite edge ideals having regularity 3. The second one, analogous to Theorem~\ref{nuestro}, gives some extra information when the bipartite edge ideal $I(G)$ has regularity $>3$: we determine the first step $i$ in the minimal graded free resolution of $I(G)$ where there are syzygies contributing to a graded Betti number located outside the first two rows of the Betti diagram. We also show that these syzygies are then concentrated in degree $i+4$ and compute the corresponding graded Betti number $\beta_{i,i+4}$.

\begin{theorem}\label{characterization}
Let $G$ be a connected bipartite graph.
The edge ideal $I(G)$ has regularity 3 if and only if
$G^{c}$ has at least one induced cycle {\rm(}of length $\geq 4${\rm)}
and $G^{bc}$ does not have any induced cycle of length $\geq 6$.
\end{theorem}

\begin{theorem} \label{main1}
Let $G$ be a connected bipartite graph and set $r:=|V(G)|$. Assume that
$\reg(I(G))>3$ and let $t=2s\geq 6$ be the minimal length of an induced cycle in $G^{bc}$. Then:
\begin{itemize}
\item $\beta_{i,j}(I(G))=0$ for all $i<t-4$ and $j>i+3$;
\item $\beta_{t-4,t}(I(G))=|\{$induced $t$-cycles in $G^{bc}\}|$;
\item $\beta_{t-4,j}(I(G))=0$ for all $j>t$;
\item
for any $\mathbf m\in {\mathbb N}^r$ such that $\vert{\mathbf m}\vert=t$, one has that
$\beta_{t-4,\mathbf m}(I(G))=1$ if $\mathbf m\in \{0,1\}^{r}$ and $G^{bc}[W]\simeq C_{t}$ where $W:=\{v_i\in V(G)\,:\ m_i=1\}$. Otherwise, $\beta_{t-4,\mathbf m}(I(G))=0$.
\end{itemize}
\end{theorem}

Before we prove these results, let's recall
a construction and some results from \cite{dalili_dependence_2010} that will be useful.
Given a simplicial complex $\Gamma$ on the vertex set $X=\{x_{1},\ldots, x_{n}\}$ whose facets are denoted by
$F_1,\ldots,F_m$, consider $m$ new vertices,
$Y:=\{y_{1},\ldots, y_{m}\}$, and define a new simplicial complex, $\Delta(\Gamma)$,
on the vertex set $X\sqcup Y$ by
\begin{equation}\label{DK}
\Delta(\Gamma) :=\Delta'\cup \Delta_{X}\,,
\end{equation}
where $\Delta_X$ denotes the $(n-1)$-simplex on the vertex set $X$, and $\Delta'$ is the simplicial complex
given by
$
\Delta'=\{{\sigma \cup \tau : \sigma \in \Gamma, \tau\subset \{{y_{j}: \sigma \subset F_j}}\} \}
$.
Then, \cite[Theorem 4.7]{dalili_dependence_2010} states that
\begin{equation}\label{teoremaDK}
\tilde H_{i+1}(\Delta(\Gamma))\simeq\tilde H_{i}(\Gamma)\,,\ \forall\,i\geq 0\,.
\end{equation}

\medskip

Let $G$ be a connected bipartite graph on the vertex set {$V(G)=X\sqcup Y$} with
$X=\{x_1,\ldots,x_n\}$, $Y=\{y_1,\ldots,y_m\}$, set
$R:=\K[x_{1},\ldots, x_{n},y_{1},\ldots, y_{m}]$, and denote as before
$W_X:=W\cap X$, $W_Y:=W\cap Y$ for any subset $W$ of $V(G)$.
One has that the set
\begin{equation}\label{DKgamma}
\Gamma_G:=\{\sigma\subset N_{G^{bc}}(y) : y\in Y\}
\end{equation}
is a simplicial complex on $X\setminus \{x\in X\hbox{ that are isolated vertices of } G^{bc}\}$.

\begin{definition}\label{reducedconstruction}
We say that a subset $W\subset V(G)$ is {\it relevant} if
$|W|\geq 3$ and, for all $u,v\in W$, $N_{G[W]}(u)\not\subset N_{G[W]}(v)$.
\end{definition}

\begin{remark}\label{rmkRelevantSubsets}
\begin{itemize}
\item
If $W$ is not relevant, then there exist $u,v\in W$ such that $N_{G[W]}(u)\subset N_{G[W]}(v)$
and $H_i(\Delta(G[W]))\simeq H_i(\Delta(G[W\setminus\{v\}]))$ for all $i\geq 0$ by Lemma~\ref{lanota}.\ref{lacinco}.
\item
If $W$ is relevant, then $|W_X|>1$ and $|W_Y|>1$.
\end{itemize}
\end{remark}

\begin{lemma}\label{facetsDK}
If $W\subset V(G)$ is relevant, then
$\Delta(\Gamma_{G[W]})=\Delta(G)[W]$.
\end{lemma}

\begin{proof}
Denote by $\Gamma:=\Gamma_{G[W]}$ the simplicial complex associated to the graph $G[W]$ as in (\ref{DKgamma}), let
$\mathcal F(\Gamma)$ be its set of facets, and set $\Delta:=\Delta(\Gamma)$ as defined in (\ref{DK}).
Since $W$ is relevant, $G[W]$ has no isolated vertex and hence $W_X$ is the vertex set of $\Gamma$.
Moreover,
$\mathcal F(\Gamma)=\{N_{G^{bc}[W]}(y) : y\in W_Y\}$.
This implies that
$\Delta=\Delta'\cup\Delta_{W_X}$ where
$\Delta'=\{\sigma\cup \tau : \sigma\in \Gamma_{G[W]}, \tau\subset\{y\in W_Y : \sigma\subset N_{G^{bc}[W]}(y)\}\}$.
Consider $\sigma\subset W$.
If $\sigma\subset W_X$ then $\sigma \in \dd_{W_X}\subset\dd$ and also $\sigma \in \dd(G)[W]$. Otherwise, one has that
 \begin{eqnarray*}
        \sigma\in\Delta\setminus\dd_{W_X}  & \Leftrightarrow & \sigma_X \in \Gamma_{G[W]}, \sigma_Y\not=\emptyset  \hbox{ and }       	 \sigma_X\subset N_{G^{bc}[W]}(y)\, , \forall y\in\sigma_Y \\
                                   & \Leftrightarrow & \sigma_Y\not=\emptyset \hbox{ and } \sigma_X\subset N_{G^{bc}[W]}(y)\, , \forall y\in\sigma_Y \\
                                   & \Leftrightarrow & \sigma_Y\not=\emptyset \hbox{ and } \{x,y\}\not\in E(G[W])\, , \forall x\in\sigma_X\, ,\forall y\in\sigma_Y \\
                                   & \Leftrightarrow & \sigma\not\subset W_X \hbox{ and } \sigma\in\Delta(G)[W]\\
                                   & \Leftrightarrow & \sigma\in\Delta(G)[W]\setminus \dd_{W_X}.\\
 \end{eqnarray*}
Thus, $\sigma\in\Delta \Leftrightarrow  \sigma\in\Delta(G)[W]$.
\end{proof}

\begin{lemma}\label{gradogens}
Let $I:=I_{\Gamma_{G[W]}}$ be the Stanley-Reisner ideal associated to the simplicial complex $\Gamma_{G[W]}$ and let
$\{m_1,\ldots,m_s\}$ be its monomial minimal generating set. One has that:
\begin{itemize}
 \item if $W$ is relevant then $\deg(m_i)\geq 2, \forall i\in[s]$;
 \item $\max\{\deg(m_i) : i\in [s]\} \leq \mu(G[W])$.
\end{itemize}
\end{lemma}

\begin{proof}
If $W$ is relevant and $x\in W$
then $N_{G[W]}(x)\subsetneq W_Y$. Thus, $x\in N_{G^{bc}[W]}(y)$ for some
$y\in W_Y$, and hence $\{x\}\in\Gamma_{G[W]}$. Therefore, non-faces must have dimension strictly greater than 1.
Since minimal generators of $I$ correspond to minimal non-faces of $\Gamma_{G[W]}$,
the first claim follows.

\smallskip
If $g=x_{i_1}\cdots x_{i_d}$ is a minimal generator of $I$, then $\{x_{i_1},\ldots,x_{i_d}\}\not\subset N_{G^{bc}[W]}(y)$ for all $y\in W_Y$ and $\frac{g}{x_{i_k}}\not\in I$ for all $k\in [d]$. Hence, for every $l\in [d]$, $F_l:=\{ x_{i_k} : k\not= l\}\subset N_{G^{bc}[W]}(y)$ for some element $y$ in $W_Y$ that we denote by
$y(l)$. Then, $x_{i_l}\not\in N_{G^{bc}[W]}(y(l)) $, or equivalently, $x_{i_l}\in N_{G[W]}(y(l)) $ and $x_{i_k}\not\in N_{G[W]}(y(l))$ if $k\not=l$. So $\{\{x_{i_l},y(l)\}:l\in [d]\}$ is a set consisting of $d$ disconnected edges of $G[W]$. This implies that $\mu(G[W])\geq d$.
\end{proof}

\begin{proof}[Proof of Theorems~\ref{characterization} and \ref{main1}]
We will first prove the equivalence in Theorem~\ref{characterization} and show that the extra information
contained in Theorem~\ref{main1} then follows quite easily.
First assume that $\reg(I(G))=3$. By Theorem~\ref{froberg},
$G^c$ contains an induced cycle of length $l\geq 4$.
Moreover, if there exists a subset $W$ of $V(G)$ such that $G^{bc}[W]\simeq C_{l}$ for some (even) $l\geq 6$ then,
since on one hand $\beta_{i,j}(I(G))\geq \beta_{i,j}(I(G[W]))$ for all $i,j$
by Hochster's Formula (\ref{hochster}), and on the other $\beta_{l-4,l}(I(C_{l}^{bc}))=1$ by
Theorem~\ref{bettibcciclo}, one gets that
$\reg(I(G))\geq 4$, a contradiction.

\smallskip
Conversely, assume that $\reg(I(G))\not=3$. If $\reg(I(G))=2$ then there is no induced cycle in $G^c$ by Theorem~\ref{froberg} and the result holds.
If $\reg(I(G))>3$ then, by Theorem~\ref{BettiDiagEdgeIdeal}, there exists $i$ such that $\beta_{i,i+4}(I)\not=0$.
Denote by $i_4$ the smallest integer with this property.
By \cite[Lemma 2.2]{katzman_characteristic-independence_2006}, $i_4\geq 2$ and if $i_4=2$, then
$\beta_{2,6}(I)\neq 0$ is the number of induced subgraphs of $G$ isomorphic to $3K_2$.
We only have to notice that $(3K_2)^{bc}\simeq C_6$ to obtain that if $i_4=2$,
$G^{bc}$ contains an induced cycle of length 6.
On the other hand, all the items in Theorem \ref{main1} follow in this case from Theorem~\ref{BettiDiagEdgeIdeal}
and \cite[Theorem~2.1]{gasharov_resolutions_2002} which states that,
for any monomial in $R$, $\xx{\mathbf m}$, if one collects at each step of the
minimal multigraded free resolution of $I(G)$, the minimal generators whose multidegree divides $\xx{\mathbf m}$,
one gets a minimal multigraded free resolution of $I(G)_{\mathbf m}$, the edge ideal whose minimal generators divide $\xx{\mathbf m}$.

\smallskip
If $i_4\geq3$, Hochster's Formula (\ref{hochsterMult}) tells us that
there exists $W\subset V(G)$ such that
\begin{equation}\label{propW}
 |W|=i_4+4 \hbox{ and } \dim_{\K}(\tilde H_{2}(\Delta(G)[W]))>0\, .
\end{equation}
As in the case $i_4=2$, we will be done using Theorem~\ref{BettiDiagEdgeIdeal} and \cite[Theorem~2.1]{gasharov_resolutions_2002} if we show that the subsets $W\subset V(G)$ satisfying (\ref{propW}) are the ones such that
\begin{equation}\label{propC}
 G[W]\simeq (C_{i_4+4})^{bc}\, .
\end{equation}
If $W$ satisfies (\ref{propC}), then it satisfies (\ref{propW}) by Proposition~\ref{homDelta}.
Now take $W$ satisfying (\ref{propW}) and consider the simplicial complex
$\Gamma:=\Gamma_{G[W]}$.
Note that, using Remark~\ref{rmkRelevantSubsets}, $W$ has to be a relevant subset of vertices by minimality of $i_4$. Applying
Lemma~\ref{facetsDK} and (\ref{teoremaDK}), one has that
$$\dim_{\K}(\tilde H_{1}(\Gamma))=\dim_{\K}(\tilde H_{2}(\Delta(G)[W]))>0\, .$$
Moreover, $\dim_{\K}(\tilde H_{1}(\Gamma[X']))=0$ for all $X'\subsetneq W_X$ since $\Delta(\Gamma_{X'})\simeq \Delta(G)[W']$ where $W'=X'\sqcup W_{Y}$ and if $\dim_{\K}(\tilde H_{1}(\Gamma[X']))=\dim_{\K}(\tilde H_{2}(\Delta(G)[W']))>0$, we will reach a contradiction with the minimality of the size of $W$.

\smallskip
As $i_4>2$, we have $\beta_{2,6}(I(G))=0$ and hence, by \cite[Lemma 2.2]{katzman_characteristic-independence_2006}, $\mu(G)=2$.
Thus, by Lemma \ref{gradogens}, $I_\Gamma$ is generated in degree 2, i.e., it is an edge ideal, and hence we can write $\Gamma=\Delta(G^{*})$ for some simple graph $G^{*}$ on the vertex set $W_X$. Thus, $\dim_{\K}(\tilde H_{1}(\Delta(G^*)))>0$ and $\dim_{\K}(\tilde H_{1}(\Delta(G^*)[X']))=0$ for all $X'\subsetneq W_X$. Applying Theorem~\ref{nuestro},
we have that $C_l<(G^*)^c$ for some $l\geq4$ but $C_l\not<(G^*)^c[X']$
for all $X'\subsetneq W_X$, so necessarily, $(G^*)^c=C_l$ and $l=|W_X|$. Therefore, $\Gamma=\dd(C_l^c)=C_l$
and we have $|N_{G^{bc}[W]}(y)|=2$, for all $y \in W_{Y}$.
Together with the fact that $N_{G^{bc}[W]}(u)\not\subset N_{G^{bc}[W]}(v)$ for all $u,v\in W$ such that $u\not= v$ (so $|N_{G^{bc}[W]}(u)|\not=1, u \in W_{X}$) and that $\sum_{u\in W_{X}}\deg_{G^{bc}[W]}(u)=\sum_{v\in W_{Y}}\deg_{G^{bc}[W]}(v)$ (so $N_{G^{bc}[W]}(u)|\leq 2, u \in W_{X})$, this implies that $|N_{G^{bc}[W]}(y)|=2$ for all $y \in W$. Moreover, $G^{bc}[W]$ is connected because $\Gamma$ is,
and hence $\bc G[W]\simeq C_{|W|}$.
\end{proof}

\begin{remark}
One can find in \cite{katzman_characteristic-independence_2006} several examples of edge ideals
whose regularity is 3 or 4 depending on the characteristic of the field $\K$. This
shows that in Theorem~\ref{characterization} the bipartite hypothesis can not be removed
since the information provided there only depends
on the combinatorics of the graph $G$. That is why we restricted ourselves to bipartite edge ideals in this work.
Now observe that even for bipartite edge ideals, it is hopeless to try an
extrapolation of our results for higher values of the regularity as an example
in \cite{dalili_dependence_2010} shows.
\end{remark}

 \section{The non-squarefree case}\label{nonsqf}

Let $I$ be an ideal in $R:=\K[x_1,\ldots,x_n]$ generated by monomials of degree two which is not squarefree.
Assume, without loss of generality, that
$I$ is minimally generated by $\{m_{1},\ldots,m_{s}\}$ where
$m_{1}=x_{1}^{2}$, \ldots, $m_{l}=x_{l}^{2}$ and $m_{l+1},\ldots,m_{s}$ are squarefree
for some $l \in [s]$.
We define
 \begin{itemize}
\item $I_{sqf}:=(m_{l+1},\ldots,m_{s})\subset R,$ and
\item $I_{pol}:=(x_{1}y_{1},\ldots,x_{l}y_{l},m_{l+1},\ldots,m_{s})\subset R^{*}:=\K[x_{1},\ldots,x_{n},y_{1},\ldots,y_{l}]$.
\end{itemize}
The ideal $I_{pol}$, called the {\it polarization} of $I$, has the following useful property: if we provide $R$ and $R^{*}$ with a $\mathbb N^{n}$-multigrading such that $\deg(x_{i})={\mathbf e_{i}}$ for all $i\in [n]$ and $\deg(y_{j})={\mathbf e_{j}}$
for all $j\in [l]$ then, by \cite[Corollary~1.6.3]{herzoghibi},
\begin{equation}\label{pol}
\beta_{i,\mathbf m}(I)=\beta_{i,\mathbf m}(I_{pol}),\ \forall i\geq 0,\ \forall \mathbf m\in \mathbb N^{n}.
\end{equation}

Both ideals $I_{sqf}$ and $I_{pol}$ are edge ideals, the first one on the vertex set $\{x_{1},\ldots,x_{n}\}$ and the second on $\{x_{1},\ldots,x_{n},y_{1},\ldots,y_{l}\}$.
We will call $G$ the non simple graph associated to $I$ and denote, as in the squarefree case, $I=I(G)$.
Denote by $G_{sqf}$ and $G_{pol}$ the simple graphs associated to $I_{sqf}$ and $I_{pol}$, respectively. Observe that $G_{sqf}$ and $G_{pol}$ are obtained by removing loops in $G$ and substituting whiskers for loops in $G$, respectively.

\begin{definition}
We say that two edges $e_{1},e_{2}\in E(G)$ are \textit{totally disjoint} provided $\{u,v\}\not\in E(G)$ if $u\in e_1$ and $v\in e_2$.
\end{definition}

\smallskip
Assume that the simple graph $G_{sqf}$ is connected and bipartite. In this case, we say that the non simple graph $G$ is {\it bipartite} and define the {\it bipartite complement} of $G$ as the bipartite complement of the simple graph $G_{sqf}$, i.e., $G^{bc}:=(G_{sqf})^{bc}$.
We also define the {\it complement} of $G$ as the complement of the simple graph $G_{sqf}$, i.e., $G^{c}:=(G_{sqf})^{c}$.

\smallskip

We can complete the characterization of ideals associated to bipartite graphs having regularity 3 with the non-squarefree case as follows:
\begin{proposition}
Let $I\subset R$ be a non-squarefree monomial ideal generated in degree two and assume that the non simple graph $G$ associated to $I$ is bipartite. Then, $I$ has regularity 3 if and only if
\begin{itemize}
 \item $G$ either has two totally disjoint edges or $C_l<G^{c}$ for some $l\geq 5$,
 \item $G$ does not have three edges that are pairwise totally disjoint, and
 \item $G^{bc}$ has no induced cycle of length $\geq 8$.
\end{itemize}

\end{proposition}

\begin{proof}
By (\ref{pol}), $\reg(I)=3$ if and only if $\reg(I_{pol})=3$ and, using Theorem~\ref{characterization}, this occurs if and only if
$(G_{pol})^{c}$ has an induced cycle of length 4 and $(G_{pol})^{bc}$ has no induced cycle of length $\geq 6$.
Rewriting these properties of the graph $G_{pol}$ in terms of the graph $G$, the result follows.
\end{proof}

When $\reg(I)>3$, the claims in Theorem~\ref{main1} remain valid if $G$ does not contain three edges that are pairwise totally disjoint since $l$-cycles in $(G_{sqf})^c$ and in $(G_{sqf})^{bc}$ coincide with the $l$-cycles in $(G_{pol})^c$ and $(G_{pol})^{bc}$ respectively, provided $l>6$. However, if $G$ has three edges that are pairwise totally disjoint, then:
 \begin{enumerate}
\item[$\bullet$] $\beta_{i,j}=0$ if $i\leq1$ and $j>i+3$;
\item[$\bullet$] $\beta_{2,6}$ is the number of induced subgraphs of $G$ isomorphic to three pairwise totally disjoint edges;
\item[$\bullet$] $\beta_{2,j}=0$ for all $j>6$;
\item[$\bullet$] considering the $\mathbb N^{n}$-multigrading on $R$, for all ${\mathbf m}\in \mathbb N^{n}$ such that $|{\mathbf m}|=6$, one has that
$\beta_{2,{\mathbf m}}=1$ if $G[\{x_{i}: m_{i}=1\}]$ consists of three totally disjoint edges. Otherwise,
$\beta_{2,{\mathbf m}}=0$.
\end{enumerate}

 \begin{example}
The ideal $I=(x_{1}^{2},x_{1}x_{5},x_{2}x_{5},x_{2}x_{7},x_{3}x_{5},x_{3}x_{6},x_{3}x_{7},x_{4}x_{6})$
satisfies that $\beta_{2,6}=1$. The bipartite graph $G^{bc}$ does not have any
induced $6$-cycle but there are three pairwise disjoint edges in $G$.
\end{example}

\end{document}